\documentclass[12pt]{article}

\usepackage{amsmath,amssymb,amsthm,amsfonts,mathtools,dsfont,hyperref,
	fullpage,xcolor,indentfirst}
\usepackage[english]{babel}

\hypersetup{colorlinks,linkcolor={red!50!black},citecolor={blue!50!black}}

\numberwithin{equation}{section}

\newtheorem{theorem}{Theorem}[section]

\newtheorem{corollary}{Corollary}[section]

\theoremstyle{remark}

\newtheorem{remark}{Remark}[section]

\DeclareMathOperator{\supp}{supp}
\DeclareMathOperator{\card}{card}

\begin{document}
	
\title{Limit theorems\\in the extended coupon collector's problem}
	
\author{Andrii Ilienko\\
	{\small\it Igor Sikorsky Kyiv Polytechnic Institute}\\
	{\small\it ilienko@matan.kpi.ua}}
	
\date{\today}
	
\maketitle
	
\begin{abstract}
	We consider an extended variant of the classical coupon collector's problem with infinite number of collections. An arriving coupon is placed in the $r^{th}$ collection, $r\ge0$, if $r$ is the smallest index such that the corresponding collection still does not have a coupon of this type. We derive distributional limit theorems for the number of empty spots in different collections at the time when the $0^{th}$ collection was completed, as well as after some delay. We also obtain limiting distributions for completion times of different collections. All main results are given in an ultimate infinite-dimensional form in the sense of distributional convergence in $\mathbb R^\infty$. The main tool in the proofs is convergence of specially constructed point processes.
\end{abstract}

\section{Introduction}

The coupon collector’s problem is undoubtedly one of the most popular classical problems in combinatorial probability. It is not only of purely theoretical interest, but also finds numerous applications (see, e.g., \cite{BH}, \cite{Luko}, or \cite{Omni}). Relatively recently, an extended version of this problem has attracted considerable attention (see \cite{Foa01}, \cite{Foa03}, \cite{AOR}, and \cite{DP18}).

Following \cite{Foa03}, we give its statement as follows. A person collects coupons, each of which belongs to one of $n\in\mathbb N$ different types. The coupons arrive one by one at discrete times, the type of each coupon being equiprobable and independent of types of preceding ones. Each time the person receives a coupon which he does not yet have, he puts it in his album. Otherwise, he gives it to his younger brother. In his turn, the latter puts it in his own album, if the coupon is new for him. Otherwise, he gives it to the next younger brother, and so on. All the brothers try to complete their own collections, using the same policy. The main collector is labelled 0, the younger brother 1, the next younger brother 2, etc.

For $r\in\mathbb N_0=\mathbb N\cup\{0\}$, let $T_r^{(n)}$ stand for the time the $r^{th}$ person completes his collection. By a classical result due to Erd\H{o}s and R\'enyi \cite{ER},
\begin{gather}
\mathbb ET_r^{(n)}=n\ln n+rn\ln\ln n+
(\gamma-\ln r!)n+{\scriptstyle\mathcal{O}}(n),\nonumber\\
\lim_{n\to\infty}\mathbb P\Bigl\{\frac{T_r^{(n)}}n-\ln n-r\ln\ln n<x\Bigr\}=\exp\Bigl(-\frac{\mathrm e^{-x}}{r!}\Bigr),\quad x\in\mathbb R,\label{ER}
\end{gather}
with $\gamma=-\Gamma'(1)$ standing for the Euler–Mascheroni constant. So, the limiting distribution is of Gumbel type.

Denote by $U_r^{(n)}$, $r\in\mathbb N$, the number of empty spots in the album of the $r^{th}$ brother at time $T_0^{(n)}$, that is, when the main collector completed his album. In \cite{Pint}, by using the optional stopping theorem for a specially constructed martingale, it was shown that
$\mathbb EU_1^{(n)}=\mathcal H_1^{(n)}$, where
$\mathcal H_1^{(n)}=\sum_{k=1}^n\frac 1k$ stands for the $n^{th}$ harmonic number. Subsequently, it turned out that this can be further generalized to
\begin{equation}
\label{EU}
\mathbb EU_r^{(n)}=\mathcal H_r^{(n)},\qquad n,r\in\mathbb N,
\end{equation}
where $\mathcal H_r^{(n)}$ are hyperharmonic numbers defined by the recursive formula
\[\mathcal H_r^{(n)}=\sum_{k=1}^n\frac {\mathcal H_{r-1}^{(k)}}k,\qquad r\ge2.\]
This was proved in \cite{Foa01} and \cite{Foa03} by using a non-elementary generating function argument, and in \cite{AOR} by a simpler and more direct probabilistic reasoning. Moreover, the latter argument made it possible to obtain another, more explicit representation:
\[\mathcal H_r^{(n)}=\sum_{k=1}^n\binom nk\frac{(-1)^{k+1}}{k^r},\qquad r\ge1.\]
By (10.1) in \cite{Foa01}, for the above hyperharmonic numbers the following asymptotic formula holds:
\begin{equation}
\label{hn_as}
\frac{r!\,\mathcal H_r^{(n)}}{\ln^rn}\to1\quad\text{as $n\to\infty$.}
\end{equation}
It should be noted that the term \lq\lq hyperharmonic numbers\rq\rq\ often refers to a different numerical sequence (see, e.g., the corresponding article in Wikipedia).

Note that this problem allows for different, though equivalent, formulations. For instance, we give another one in terms of a balls-into-bins model. Suppose we have an infinite sequence of balls and $n$ bins of unlimited capacity. Each time a single ball is placed into one of the bins, which is chosen equiprobably and independently of the previous history. In this setting, $T_r^{(n)}$ means the first time when each bin contains at least $r+1$ balls. Similarly, $U_r^{(n)}$ stands for the number of bins containing at most $r$ balls at time $T_0^{(n)}$, that is, when each bin first contains at least one ball.

Throughout all this time, the question about limiting distributions for $U_r^{(n)}$ as $n\to\infty$ remained open. This problem was explicitly stated on p.~446 in \cite{DP18}. In order to determine these limiting distributions it would be natural to use the generating function of $U_r^{(n)}$. The latter can be easily deduced from Proposition 4.1 in \cite{Foa01}:
\[\mathbb Eu^{U_r^{(n)}}=nu\sum_{\substack{a+b+c_1+\\+\ldots+c_r=n-1}}
\binom{n-1}{a,b,c_1,\ldots,c_r}(-1)^b(u-1)^{\sum_{k=1}^rc_k}
\frac{(\sum_{k=1}^rkc_k)!}{\prod_{k=1}^r(k!)^{c_k}}(n-a)^{-1-\sum_{k=1}^rkc_k}.\]
It is, however, clear that the generating function of such an ominous form can hardly be used for this purpose. Therefore, as in our previous paper \cite{Il19}, we take a different approach based on the use of point processes. This allows not only to determine the limiting distributions for $U_r^{(n)}$ (looking ahead, they will turn out to be exponential) but also to obtain the limit theorem in an ultimate infinite-dimensional form. In other words, we will prove the distributional convergence of the sequence $\bigl(U_r^{(n)}, r\in\mathbb N\bigr)$, normalized in a proper way, to a random element $(E,E,\ldots)$ of $\mathbb R^{\infty}$ as $n\to\infty$, where $E$ is (all the time the same!) $\mathsf {Exp}(1)$-distributed random variable. This means that $U_r^{(n)}$ exhibit an incredible asymptotic stability. The precise statement is given in Theorem \ref{main_th} below.

Note that in the proof of this limit theorem we use a rather unusual trick which, hopefully, can be applied elsewhere. Namely, appealing to a certain analogy between normalization and thinning, we prove the distributional convergence for thinned random variables rather than for normalized ones. This convergence can be deduced from that of specially constructed point processes. The way we then return to normalized random variables is described in Section \ref{th_tr}.

We also prove another limit theorem which can be called delayed one. The idea behind is as follows. Due to their definition, the random variables $U_r^{(n)}$ increase unboundedly as $n\to\infty$. Thus, in order to obtain a distributional limit, we had to use some normalization. If one still wants to obtain a limit result without any normalization, one should, instead of $U_r^{(n)}$, consider the number of empty spots in the album of the $r^{th}$ brother not at time $T_0^{(n)}$ but after some delay, that is, at time $T_0^{(n)}+d_r^{(n)}$ with some non-random $d_r^{(n)}\to\infty$ as $n\to\infty$. With the right choice of $d_r^{(n)}$, the corresponding numbers of empty spots will not go to infinity as $n$ does, but will instead converge in distribution to some limit law. In Theorem \ref{second_th}, we find this right form of $d_r^{(n)}$ and establish convergence to a geometric distribution. Like the previous result, we state and prove this one in an infinite-dimensional form, that is, in the sense of distributional convergence in $\mathbb R^\infty$.

Our third main result, Theorem \ref{third_th} below, describes the asymptotic behaviour of the sequence of completion times, that is, of the (centered and normalized in a proper way) random elements $\bigl(T_r^{(n)}, r\in\mathbb N_0\bigr)$ in $\mathbb R^\infty$. Actually, this is an infinite-dimensional extension of \eqref{ER}, the limit theorem due to Erd\H{o}s and R\'enyi. Surprisingly, the limiting random element turns out to consist of independent entries, which means that the completion times $T_r^{(n)}$, after appropriate centering and normalizing, become asymptotically independent. This may seem all the more unexpected since $T_r^{(n)}$ increase in $r$, which excludes any kind of independence for non-centered completion times. So, this asymptotic independence is achieved only at the expense of centering and normalizing.
As a direct application of the latter result, we derive a limit theorem for inter-completion times with limiting logistic distribution.

\section{Preliminaries and main results}
\label{prel_sec}

Denote by $Y_{i,r}^{(n)}$, $n\in\mathbb N$, $i\le n$, $r\in\mathbb N_0$, the arrival time of the $(r+1)^{th}$ coupon of type $i$. So, at time $Y_{i,r}^{(n)}$ the collector labelled $r$ receives such a coupon into his collection. Hence,
\[\mathbb P\bigl\{Y_{i,r}^{(n)}=k\bigr\}=
\binom{k-1}{r}\Bigl(\frac 1n\Bigr)^{r+1}\Bigl(1-\frac 1n\Bigr)^{k-r-1},\quad k\ge r+1,\]
which is one of versions of the negative binomial distribution, namely, the one that counts trials up to the $(r+1)^{th}$ success.
Note that
\begin{gather}
\label{T}
T_r^{(n)}=\max_{i\le n}Y_{i,r}^{(n)},\\
\label{UY}
U_r^{(n)}=\card\bigl\{i\colon Y_{i,r}^{(n)}>
\max_{i\le n}Y_{i,0}^{(n)}\bigr\}.
\end{gather}

For different $i$, the random variables $Y_{i,r}^{(n)}$ are identically distributed but not independent. There is a standard way to get rid of this dependence --- a poissonization trick due to Holst (see \cite{Hol86} or \cite{Il19} for details). Namely, consider a poissonized scheme, that is, assume that coupons arrive at random times with independent $\mathsf{Exp}(1)$-distributed intervals $E_j$, $j\in\mathbb N$. Similarly to the above, denote by $Z_{i,r}^{(n)}$ the arrival time of the $(r+1)^{th}$ coupon of type $i$ in the poissonized scheme. Then, for different $i$, $Z_{i,r}^{(n)}$ are independent gamma-distributed random variables, $Z_{i,r}^{(n)}\sim\mathsf{\Gamma}\bigl(r+1,\frac 1n\bigr)$, which is a consequence of the thinning theorem for Poisson point processes (see, e.g., Theorem 5.8 in \cite{LP}). For any fixed $n$, the sequences $\bigl(Y_{i,r}^{(n)}\bigr)$ and $\bigl(Z_{i,r}^{(n)}\bigr)$ can be given on a common probability space and coupled by
\begin{equation}
\label{ZY}
Z_{i,r}^{(n)}=\sum_{j=1}^{Y_{i,r}^{(n)}}E_j,\quad i\le n,\,r\in\mathbb N_0.
\end{equation}
Moreover, $\bigl(Y_{i,r}^{(n)},i\le n, r\in\mathbb N_0\bigr)$ is independent of $(E_j, j\in\mathbb N)$, because the inter-arrival times $E_j$ do not affect the order in which the coupons of different types arrive. Finally, note that, along with \eqref{UY}, a similar formula in terms of $Z_{i,r}^{(n)}$ holds:
\begin{equation}
\label{UZ}
U_r^{(n)}=\card\bigl\{i\colon Z_{i,r}^{(n)}>
\max_{i\le n}Z_{i,0}^{(n)}\bigr\}.
\end{equation}

We now state our first main result, an infinite-dimensional distributional limit theorem for $\bigl(U_r^{(n)},r\in\mathbb N\bigr)$ in the space $\mathbb R^{\infty}$.

\begin{theorem}
\label{main_th}
Let $E\sim\mathsf{Exp}(1)$. Then
\begin{equation}
\label{main_th_eq}
\biggl(\frac{r!\,U_r^{(n)}}{\ln^rn},r\in\mathbb N\biggr)\xrightarrow d(E,E,\ldots)\quad\text{ in }\mathbb R^{\infty}\text{ as }n\to\infty.
\end{equation}
\end{theorem}

Let us now turn to the delayed limit theorem. By analogy with $U_r^{(n)}$, for fixed $d_r^{(n)}\in\mathbb N$ denote by $\hat U_r^{(n)}$ the number of empty spots in the album of the $r^{th}$ brother at time $T_0^{(n)}+d_r^{(n)}$:
\begin{equation}
\label{hU}
\hat U_r^{(n)}=\card\bigl\{i\colon Y_{i,r}^{(n)}>
\max_{i\le n}Y_{i,0}^{(n)}+d_r^{(n)}\bigr\}.
\end{equation}
The following result asserts that, with the right choice of $d_r^{(n)}$, the random elements $\bigl(\hat U_r^{(n)},r\in\mathbb N\bigr)$ of $\mathbb R^\infty$ converge in distribution to some limiting random element with geometrically distributed marginals.

\begin{theorem}
\label{second_th}
Let
\begin{equation}
\label{d_r}
d_r^{(n)}=rn\ln\ln n+{\scriptstyle\mathcal{O}}(n),\quad r\in\mathbb N,
\end{equation}
where ${\scriptstyle\mathcal{O}}(n)$, of course, may differ for different $r$.
Let also $(G_r,r\in\mathbb N)$ be a random element of $\mathbb R^\infty$ with the probability generating function of the following form:
\begin{equation}
\label{Gen_G}
\mathbb E\Bigl(\prod_{r=1}^\infty u_r^{G_r}\Bigr)=
\Bigl(\mathrm e-\sum_{r=1}^\infty\frac{u_r}{r!}\Bigr)^{-1},\quad (u_r,r\in\mathbb N)\in[0,1]^\infty.
\end{equation}
Equivalently, $(G_r,r\in\mathbb N)$ may be defined as the random sequence
$\bigl(N_r(E/r!),r\in\mathbb N\bigr)$, where $E\sim\mathsf{Exp}(1)$ and $N_r$ stand for unit-rate Poisson counting processes, independent of each other and of $E$.

Then
\begin{equation}
\label{second_th_eq}
\bigl(\hat U_r^{(n)},r\in\mathbb N\bigr)\xrightarrow
d(G_r,r\in\mathbb N)\quad\text{ in }\mathbb R^{\infty}\text{ as }n\to\infty.
\end{equation}
\end{theorem}

\begin{remark}
It follows from \eqref{Gen_G} that the marginal probabaility generating functions of the limiting random element take the form:
\[\mathbb Eu^{G_r}=
\frac{r!}{r!+1-u},\quad u\in[0,1].\]
This implies that $G_r\sim\mathsf{Geom}\bigl(\frac{r!}{r!+1}\bigr)$, that is,
\[\mathbb P\{G_r=k\}=\Bigl(\frac 1{r!+1}\Bigr)^k\frac{r!}{r!+1},\quad k\in\mathbb N_0.\]

It is worth noting that various sums of $G_r$ also follow geometric distribution. Indeed, let $I\subset\mathbb N$ and $S_I=\sum_{r\in I}G_r$. Then by \eqref{Gen_G},
\[
\mathbb Eu^{S_I}=\mathbb E\prod_{r\in I}u^{G_r}=
\biggl(\mathrm e-\sum_{r\notin I}\frac 1{r!}-
u\sum_{r\in I}\frac 1{r!}\biggr)^{-1}=
\biggl(1+\sum_{r\in I}\frac 1{r!}-
u\sum_{r\in I}\frac 1{r!}\biggr)^{-1},\quad u\in[0,1].
\]
This means that $S_I\sim\mathsf{Geom}(P_I)$ with
$P_I=\bigl(1+\sum_{r\in I}\frac 1{r!}\bigr)^{-1}$. In particular,
\[\sum_{r=1}^\infty G_r=S_{\mathbb N}\sim\mathsf{Geom}\bigl(\mathrm e^{-1}\bigr).\]
\end{remark}

\begin{remark}
The limiting random sequence $(G_r,r\in\mathbb N)$ allows for an interpretation in terms of a balls-into-bins model. Consider an infinite set of bins of unlimited capacity, numbered $0,1,2,\ldots$ Each time a ball is placed into one of the bins, choosing the $r^{th}$ one with probability
\begin{equation}
\label{p_r}
p_r=\frac 1{\mathrm er!},\quad r\in\mathbb N_0,
\end{equation}
independently of the previous choices. In other words, the numbers of consecutive bins form an i.i.d.~sequence of $\mathsf{Pois}(1)$-distributed random variables. This process continues until a ball is placed into $0^{th}$ bin. Then $G_r$, $r\in\mathbb N$, describes the number of balls in the $r^{th}$ bin.

Indeed, taking \eqref{p_r} into account, we may expand the right-hand side of \eqref{Gen_G} in a geometric series:
\begin{equation}
\label{Gen_mod}
\mathbb E\Bigl(\prod_{r=1}^\infty u_r^{G_r}\Bigr)=\sum_{k=0}^\infty
p_0\Bigl(\sum_{r=1}^\infty p_ru_r\Bigr)^k=
\sum\biggl(\frac{(k_1+k_2+\ldots)!}{k_1!\cdot k_2!\cdot\ldots}p_0\prod_{r=1}^\infty p_r^{k_r}\biggr)\biggl(\prod_{r=1}^\infty u_r^{k_r}\biggr),
\end{equation}
where the sum on the right-hand side is over the set of all infinite sequences $(k_r,r\in\mathbb N)$ of non-negative integers, only a finite number of which are non-zero. Clearly, the probability generating function on the right-hand side of \eqref{Gen_mod} corresponds to the above balls-into-bins model.

Note that finite-dimensional counterparts of infinite-dimensional geometric (and, more generally, negative binomial) distributions like that of $(G_r,r\in\mathbb N)$ were introduced and studied in \cite{BN52} and \cite{JK77}; see also a survey of various multivariate geometric and negative binomial distributions in \cite{DR96}.
\end{remark}

As our final result, we give an infinite-dimensional extension of \eqref{ER}, the limit theorem by Erd\H{o}s and R\'enyi.

\begin{theorem}
\label{third_th}
Let $B_r$, $r\in\mathbb N_0$, be independent Gumbel-distributed random variables with distribution functions
\begin{equation}
\label{Br}
\mathbb P\{B_r<x\}=\exp\Bigl(-\frac{\mathrm e^{-x}}{r!}\Bigr),\quad x\in\mathbb R.
\end{equation}
Then
\begin{equation}
\label{third_th_eq}
\biggl(\frac{T_r^{(n)}}n-\ln n-r\ln\ln n,r\in\mathbb N_0\biggr)\xrightarrow
d(B_r,r\in\mathbb N_0)\quad\text{ in }\mathbb R^{\infty}\text{ as }n\to\infty.
\end{equation}
\end{theorem}

In particular, this result allows obtaining limit distributions for times between completions of different collections. Denote by $\Delta_{r_1,r_2}^{(n)}=T_{r_2}^{(n)}-T_{r_1}^{(n)}$, $r_2>r_1$, such an inter-completion time.

\begin{corollary}
\label{cor}
Let $L_{r_1,r_2}$ be a logistic random variable with distribution function
\[\mathbb P\bigl\{L_{r_1,r_2}<x\bigr\}=\frac{r_2!}{r_2!+r_1!\,\mathrm e^{-x}},\quad x\in\mathbb R.\]
Then
\begin{equation}
\label{cor_eq}
\frac{\Delta_{r_1,r_2}^{(n)}}n-(r_2-r_1)\ln\ln n\xrightarrow dL_{r_1,r_2} \quad\text{as }n\to\infty.
\end{equation}
\end{corollary}

It is well known (and may be easily checked by means of characteristic functions) that the difference of two independent Gumbel-distributed random variables has a logistic distribution. Particularly, $B_{r_2}-B_{r_1}\overset d=L_{r_1,r_2}$.
Thus, \eqref{cor_eq} follows from \eqref{third_th_eq}. 

\section{Convergence of associated point processes I.
\\Thinned processes}
\label{sec_conv_I}

One of the key points in the proof of Theorem \ref{main_th} is convergence of specially constructed point processes to a Poisson one. We now proceed to the corresponding construction.

For fixed $n\in\mathbb N$ and $r\in\mathbb N_0$, let
\begin{gather}
\psi^{(n)}(x)=\frac xn-\ln n,\quad x\in\mathbb R,\label{psi}\\
\eta_r^{(n)}=\sum_{i=1}^n\delta_{\psi^{(n)}\bigl(Z_{i,r}^{(n)}\bigr)},\label{eta}
\end{gather}
where $\delta_u$ stands for the Dirac measure $\mathds 1\{u\in\cdot\}$. The processes $\eta_r^{(n)}$ describe $(r+1)^{th}$ arrivals of different coupon types in the poissonized scheme. Note that in a similar case in \cite{Il19}, in order to provide the necessary convergence, we used an $r$-dependent centering/normalizing function
\begin{equation}
\psi_r^{(n)}(x)=\frac xn-\ln n-r\ln\ln n,
\quad x\in\mathbb R.\label{psi_r}
\end{equation}
In our case, however, we want to consider the numbers of empty spots in different albums at the same point in time, namely when the $0^{th}$ album is completed. Hence, we have to deal with the same centering for different $r$. So, we will achieve the desired convergence in a different way --- by means of thinnings.

For a proper point process $\eta$, let us denote by $T_p\eta$, $p\in[0,1]$, its $p$-thinning, i.e.~the point process which independently keeps the points of $\eta$ with probability $p$ and removes otherwise (see, e.g., \cite{LP}, Section 5.3 for details). On the space
\begin{equation}
\label{X}
\mathbb X=\mathbb N_0\times\bigl(\mathbb R\cup\{+\infty\}\bigr),
\end{equation}
endowed with some relevant metric, say,
\begin{equation}
\label{d}
d\bigl((r_1,x),(r_2,y)\bigr)=|r_2-r_1|+
|\mathrm e^{-y}-\mathrm e^{-x}|,
\end{equation}
consider the point processes $H^{(n)}$, $n\ge3$, given as follows: for $B_r\in\mathfrak B\bigl(\mathbb R\cup\{+\infty\}\bigr)$, $r\in\mathbb N_0$, define 
\begin{equation}
\label{H}
H^{(n)}\Bigl(\bigcup_{r=0}^\infty\bigl(\{r\}\times B_r\bigr)\Bigr)=\sum_{r=0}^\infty\bigl(T_{\ln^{-r}n}\eta_r^{(n)}\bigr)(B_r).
\end{equation}
(We require $n\ge3$ in order to have $\ln^{-r}n\in[0,1]$ for all $r\in\mathbb N_0$.)
In other words, we are thinning out the processes $\eta_r^{(n)}$ and glue them into one \lq\lq multilevel\rq\rq\ point process $H^{(n)}$. Note that, by construction in \eqref{eta},
\begin{equation}
\label{empty}
\bigl(T_{\ln^{-r}n}\eta_r^{(n)}\bigr)\bigl(\{+\infty\}\bigr)=
\eta_r^{(n)}\bigl(\{+\infty\}\bigr)=0
\end{equation}
anyway. The reason we, nevertheless, consider the semi-compactified real axis $\mathbb R\cup\{+\infty\}$ instead of just $\mathbb R$ will become clear from what follows (see Section \ref{proof_sec}).

Before stating the main theorem of this section, we reformulate the basic definitions related to convergence of point processes as applied to our problem (for a detailed exposition in the abstract setting, see \cite{Res87}, \cite{Res07}, or \cite{Kal}). Let $M_p(\mathbb X)$ denote the space of all locally finite (with respect to $d$ in \eqref{d}) point measures on $\mathbb X$ given by \eqref{X}. For $\mu,\mu_1,\mu_2,\ldots\in M_p(\mathbb X)$, $\mu_n$ are said to converge vaguely to $\mu$ (denoted by $\mu_n\xrightarrow{v}\mu$) if
$\int_\mathbb X f\,\mathrm d\mu_n\to\int_\mathbb X f\,\mathrm d\mu$ for each 
continuous compactly supported non-negative test function $f$ defined on $\mathbb X$. In our case, this means that
\[
\int_{\mathbb R\cup\{+\infty\}}g(t)\,\mu_n\bigl(\{r\}\times\mathrm dt\bigr)\to\int_{\mathbb R\cup\{+\infty\}}g(t)\,\mu\bigl(\{r\}\times\mathrm dt\bigr)
\]
for each $r\in\mathbb N_0$ and each continuous compactly supported $g\colon\mathbb R\cup\{+\infty\}\to[0,+\infty)$.
As usual, the set $M_p(\mathbb X)$, equipped with the topology of the above convergence, can be metrized as a complete separable metric space. This setting allows to consider the distributional convergence of point processes $H^{(3)},H^{(4)}\ldots$, denoted as $H^{(n)}\xrightarrow{vd}H$. The main result of this section, Theorem \ref{conv_th} below, asserts that the point processes $H^{(n)}$ converge in this sense toward a non-homogeneous Poisson process.

\begin{theorem}
\label{conv_th}
Let $H$ be a Poisson point process on $\mathbb X$ with intensity measure $\lambda$ given by
\begin{equation}
\label{lambda}
\lambda\Bigl(\bigcup_{r=0}^\infty\bigl(\{r\}\times B_r\bigr)\Bigr)=
\sum_{r=0}^{\infty}\frac 1{r!}\int_{B_r}\mathrm e^{-x}\,\mathrm dx,\quad B_r\in\mathfrak B\bigl(\mathbb R\cup\{+\infty\}\bigr).
\end{equation}
Then $H^{(n)}\xrightarrow{vd}H$ as $n\to\infty$.
\end{theorem}

\begin{remark}
\label{ind_PPP}
By the superposition theorem for Poisson processes (see, e.g., Theorem 3.3 in \cite{LP}), the additive structure of $\lambda$ in \eqref{lambda} implies that different levels $H\bigl(\{r\}\times\cdot\bigr)$, $r\in\mathbb N_0$, of the limiting process $H$ are independent. In other words, the point processes $T_{\ln^{-r}n}\eta_r^{(n)}$ are asymptotically independent. This fact is rather surprising, since without thinning no asymptotic independence would have been expected. Indeed, by \eqref{psi} and \eqref{eta}, $\eta_r^{(n)}$ increase in $r$:
$\psi^{(n)}\bigl(Z_{i,r_1}^{(n)}\bigr)<\psi^{(n)}\bigl(Z_{i,r_2}^{(n)}\bigr)$ for $r_1<r_2$ and fixed $n$, $i$, which excludes any independence.
\end{remark}

\begin{remark}
\label{nhg_PPP}
The levels $H\bigl(\{r\}\times\cdot\bigr)$ allow for a simple interpretation (see Remark 3.2 in \cite{Il19}). Let $\zeta$ be a stationary unit-rate Poisson point process restricted to $(0,+\infty)$, and put
\begin{equation*}
h(x)=-\ln r!-\ln x,\quad x>0.
\end{equation*}
Then, $H\bigl(\{r\}\times\cdot\bigr)\overset d=\sum_{x\in\supp\zeta}\delta_{h(x)}$.
\end{remark}

\begin{proof}[Proof of Theorem \ref{conv_th}] Let $\mathcal U$ denote the ring of all Borel subsets in $\mathbb X$, bounded with respect to $d$ in \eqref{d}. So,
\begin{equation}
\label{bound_set}
\mathcal U=\Bigl\{\bigcup_{r=0}^s\bigl(\{r\}\times B_r\bigr),
s\in\mathbb N_0,B_r\in\mathfrak B\bigl(\mathbb R\cup\{+\infty\}\bigr)\text{ and are bounded from below}\Bigr\}.
\end{equation}
Since the measure $\lambda$ given by \eqref{lambda} is diffuse, $H$ is a simple Poisson point process (see, e.g., Proposition 6.9 in \cite{LP}). Thus, by Theorem 4.18 in \cite{Kal}, it suffices only to prove that, for each $U\in\mathcal U$,
\begin{gather}
\lim_{n\to\infty}\mathbb P\{H^{(n)}(U)=0\}=\mathbb P\{H(U)=0\},
\label{cond_a}\\
\lim_{n\to\infty}\mathbb EH^{(n)}(U)=\mathbb EH(U).
\label{cond_b}
\end{gather}
Note that, by \eqref{empty} and \eqref{lambda}, $H^{(n)}\bigl(\mathbb N_0\times\{+\infty\}\bigr)=H\bigl(\mathbb N_0\times\{+\infty\}\bigr)=0$ a.s. So, we may assume that $U\subset\mathbb N_0\times\mathbb R$.

First we take up the proof of \eqref{cond_a}. According to \eqref{bound_set} and \eqref{H},
\begin{equation}
\begin{aligned}
\label{p_eq_0}
\mathbb P\{&H^{(n)}(U)=0\}=\mathbb P\Bigl\{H^{(n)}\Bigl(\bigcup_{r=0}^s\bigl(\{r\}\times B_r\bigr)\Bigr)=0\Bigr\}\\&=\mathbb P\Bigl\{\sum_{r=0}^s \bigl(T_{\ln^{-r}n}\eta_r^{(n)}\bigr)(B_r)=0\Bigr\}=
\mathbb P\bigl\{\bigl(T_{\ln^{-r}n}\eta_r^{(n)}\bigr)(B_r)=0
\text{ for }r=0,\ldots,s\bigr\}.
\end{aligned}
\end{equation}
For $0\le r_1<\ldots<r_m\le s$, denote
\begin{equation}
\label{P}
P_{r_1,\ldots,r_m}^{(n)}=
\mathbb P\bigl\{\psi^{(n)}\bigl(Z_{i,r_1}^{(n)}\bigr)\in B_{r_1},\ldots,\psi^{(n)}\bigl(Z_{i,r_m}^{(n)}\bigr)\in B_{r_m}
\bigr\},
\end{equation}
which does not depend on $i$ by virtue of the i.i.d.~property of $Z_{i,r}^{(n)}$.
By this property again, \eqref{p_eq_0} and inclusion-exclusion yield
\begin{align*}
\mathbb P\{H^{(n)}(U)=0\}&=\Bigl(1-\sum_{0\le r_1\le s}\ln^{-r_1}n\cdot P_{r_1}^{(n)}\\&+\sum_{0\le r_1<r_2\le s}\ln^{-r_1-r_2}n\cdot P_{r_1,r_2}^{(n)}-\ldots+(-1)^{s+1}\ln^{-1-\ldots-s}n\cdot P_{0,1,\ldots,s}^{(n)}\Bigr)^n.
\end{align*}
Hence,
\begin{align*}
\lim_{n\to\infty}\ln\mathbb P\{H^{(n)}(U)=0\}=&-\sum_{0\le r_1\le s}\lim_{n\to\infty}\bigl(n\ln^{-r_1}n\cdot P_{r_1}^{(n)}\bigr)\\&+
\sum_{0\le r_1<r_2\le s}\lim_{n\to\infty}\bigl(n\ln^{-r_1-r_2}n\cdot P_{r_1,r_2}^{(n)}\bigr)-\ldots\\&+(-1)^{s+1}\lim_{n\to\infty}
\bigl(n\ln^{-1-\ldots-s}n\cdot P_{0,1,\ldots,s}^{(n)}\bigr).
\end{align*}
We now prove that the limits in the first sum equal
$\frac 1{r_1!}\int_{B_{r_1}}\mathrm e^{-x}\,\mathrm dx$,
while those in the second sum vanish. Since all subsequent terms are dominated by the latter ones, this combined with \eqref{lambda} will prove \eqref{cond_a}.

As $Z_{i,0}^{(n)},Z_{i,1}^{(n)}-Z_{i,0}^{(n)},\ldots,Z_{i,s}^{(n)}-Z_{i,s-1}^{(n)}$ are independent $\mathsf{Exp}\bigl(\frac 1n\bigr)$, the densities $f_{r_1}^{(n)}$ and $f_{r_1,r_2}^{(n)}$ of $\psi^{(n)}\bigl(Z_{i,r_1}^{(n)}\bigr)$ and $\bigl(\psi^{(n)}\bigl(Z_{i,r_1}^{(n)}\bigr),
\psi^{(n)}\bigl(Z_{i,r_2}^{(n)}\bigr)\bigr)$, respectively, can be easily calculated:
\begin{gather}\label{f_1}
f_{r_1}^{(n)}(x)=\frac {(x+\ln n)^{r_1}\mathrm e^{-x}}{n\,r_1!}\cdot\mathds 1\{-\ln n\le x\},\qquad x\in\mathbb R,\\
f_{r_1,r_2}^{(n)}(x,y)=\frac {(x+\ln n)^{r_1}(y-x)^{r_2-r_1-1}\mathrm e^{-y}}{n\,r_1!\,(r_2-r_1-1)!}\cdot\mathds 1\{-\ln n\le x\le y\},
\qquad x,y\in\mathbb R.\label{f_2}
\end{gather}
Recall that $B_{r_1}$ and $B_{r_2}$ were assumed to be bounded from below. Hence, by \eqref{P},
\begin{equation}
\begin{aligned}
\label{1-lim}
n&\ln^{-r_1}n\cdot P_{r_1}^{(n)}=n\ln^{-r_1}n\cdot\int_{B_{r_1}}f_{r_1}^{(n)}(x)\,\mathrm dx\\&=
\frac 1{r_1!}\int_{B_{r_1}}\Bigl(1+\frac x{\ln n}\Bigr)^{r_1}\mathrm e^{-x}\cdot\mathds 1\{-\ln n\le x\}\,\mathrm dx\to\frac 1{r_1!}\int_{B_{r_1}}
\mathrm e^{-x}\,\mathrm dx\quad\text{as $n\to\infty$}
\end{aligned}
\end{equation}
due to dominated convergence. Next, 
\begin{multline*}
n\ln^{-r_1-r_2}n\cdot P_{r_1,r_2}^{(n)}=n\ln^{-r_1-r_2}n\cdot\iint_{B_{r_1}\times B_{r_2}}
f_{r_1,r_2}^{(n)}(x,y)\,\mathrm dx\,\mathrm dy\\=
\frac {\ln^{-r_2}n}{r_1!\,(r_2-r_1-1)!}\iint_{B_{r_1}\times B_{r_2}}\Bigl(1+\frac x{\ln n}\Bigr)^{r_1}(y-x)^{r_2-r_1-1}\mathrm e^{-y}\cdot\mathds 1\{-\ln n\le x\le y\}\,\mathrm dx\,\mathrm dy,
\end{multline*}
which vanishes as $n\to\infty$ again by dominated convergence. This completes the proof of \eqref{cond_a}.

We now turn to the proof of \eqref{cond_b}. Similarly to \eqref{p_eq_0}, we have
\begin{equation*}
\mathbb EH^{(n)}(U)=\mathbb EH^{(n)}\Bigl(\bigcup_{r=0}^s\bigl(\{r\}\times B_r\bigr)\Bigr)\\=\sum_{r=0}^s\mathbb E\bigl(T_{\ln^{-r}n}\eta_r^{(n)}\bigr)(B_r)=
\sum_{r=0}^s\ln^{-r}n\cdot\mathbb E\eta_r^{(n)}(B_r).
\end{equation*}
It follows from \eqref{eta} and the i.i.d.~property of $Z_{i,r}^{(n)}$ that $\eta_r^{(n)}(B_r)\sim\mathsf{Bin}\bigl(n,P_r^{(n)}\bigr)$ with $P_r^{(n)}$ given by \eqref{P}. So, by \eqref{1-lim},
\[
\mathbb EH^{(n)}(U)=\sum_{r=0}^sn\ln^{-r}n\cdot P_r^{(n)}\to\sum_{r=0}^s\frac 1{r!}\int_{B_r}\mathrm e^{-x}\,\mathrm dx\quad\text{as $n\to\infty$},
\]
which equals $\mathbb EH(U)$ due to \eqref{lambda}. This concludes the proof of \eqref{cond_b} and thus of Theorem \ref{conv_th}.
\end{proof}

\section{Thinning trick}
\label{th_tr}

In this section, we consider a somewhat unusual approach to proving limit theorems, which is the basic tool in the proof of Theorem \ref{main_th}.
This approach seems to be new, and may hopefully prove to be useful elsewhere as well.

For an $\mathbb N_0$-valued random variable $X$, define its $p$-thinning, $p\in[0,1]$, by
\begin{equation}
\label{thin}
p\odot X=\sum_{i=1}^X\varepsilon_i,
\end{equation}
where $\varepsilon_i$ are $\mathsf{Bin}(1,p)$-distributed and independent of each other and of $X$. Note that, in the notation of the previous section, $(T_p\eta)(B)\overset d=p\odot\eta(B)$ for a proper point process $\eta$ and a Borel set $B$. This operation, going back to \cite{Ren}, was then used in \cite{SvH} to introduce the concepts of discrete self-decomposability and stability.

The idea behind the proposed technique is to replace the normalization by thinning. In other words, instead of proving that $p^{(n)}X^{(n)}\xrightarrow dY$ with some constants $p^{(n)}\to0$ and a limiting random variable $Y$, we will prove that $p^{(n)}\odot X^{(n)}\xrightarrow dZ$ with some new limiting random variable $Z$. The connection between the distributions of $Y$ and $Z$ can be guessed from the following examples. For $X^{(n)}=n$ a.s., we have $\frac 1n X^{(n)}=1\xrightarrow d1$ and $\frac 1n\odot X^{(n)}\xrightarrow dZ\sim\mathsf{Pois}(1)$ by the Poisson limit theorem. More generally, consider an i.i.d.~sequence $(\xi_i,i\in\mathbb N)$ with $\mathbb E\xi_i=\lambda$, and denote $X^{(n)}=\sum_{i=1}^n\xi_i$. Due to the law of large numbers, $\frac 1n X^{(n)}\xrightarrow d\lambda$, and, by Theorem 3.1 in \cite{HJK}, $\frac 1n\odot X^{(n)}\xrightarrow dZ\sim\mathsf{Pois}(\lambda)$. All this suggests that, in general, $Z$ must have the mixed Poisson distribution with mixing distribution of $Y$ (see, e.g., Chapter 2 in \cite{Gran}).
This means that
\begin{equation}
\label{univ_mix_Pois}
\mathbb P\{Z=k\}=\int_{[0,+\infty)}\frac{\mathrm e^{-y}y^k}{k!}\,F_Y(\mathrm dy),
\quad k\in\mathbb N_0,
\end{equation}
where $F_Y$ stands for the distribution function of $Y$.

The main result of this section, Theorem \ref{trick_th} below, justifies this approach. With an eye to the future, we will state and prove it in multidimensional form. Let $\mathbf Y=(Y_1,\ldots,Y_s)$ be a random vector with a.s.\ non-negative components. By analogy with \eqref{univ_mix_Pois}, the random vector $\mathbf Z=(Z_1,\ldots,Z_s)$ is said to have a multivariate mixed Poisson distribution with mixing distribution of $\mathbf Y$ (see \cite{FLT} or \cite{KP}) if
\begin{equation}
\begin{aligned}
\label{multiv_mix_Pois}
\mathbb P\{Z_1&=k_1,\ldots,Z_s=k_s\}\\&=
\idotsint_{[0,+\infty)^s}\frac{\mathrm e^{-y_1}y_1^{k_1}}{k_1!}\cdot\ldots\cdot
\frac{\mathrm e^{-y_s}y_s^{k_s}}{k_s!}\,F_{\mathbf Y}(\mathrm dy_1,\ldots,\mathrm dy_s),
\quad k_1,\ldots,k_s\in\mathbb N_0.
\end{aligned}
\end{equation}

\begin{remark}
\label{trick_rem}
In what follows, we will need another equivalent interpretation of the multivariate mixed Poisson distribution. Namely, we may define $\mathbf Z$ by $Z_r=N_r(Y_r)$, $r=1,\ldots,s$, where $N_r$ stand for unit-rate Poisson counting processes, independent of each other and of~$\mathbf Y$.
\end{remark}

\begin{theorem}
\label{trick_th}
Let $\mathbf X^{(n)}=\bigl(X_1^{(n)},\ldots,X_s^{(n)}\bigr)$, $n\in\mathbb N$, be a sequence of random vectors with $\mathbb N_0$-valued components, and $\bigl(p_1^{(n)},\ldots,p_s^{(n)}\bigr)$, $n\in\mathbb N$, a non-random sequence with $\lim_{n\to\infty}p_r^{(n)}=0$, $r=1,\ldots,s$. Assume that
\begin{equation}
\label{px_to_z}
\bigl(p_1^{(n)}\odot X_1^{(n)},\ldots,p_s^{(n)}\odot X_s^{(n)}\bigr)\xrightarrow d\bigl(Z_1,\ldots,Z_s\bigr)\quad\text{ as $n\to\infty$},
\end{equation}
and the limiting random vector on the right-hand side has a multivariate mixed Poisson distribution \eqref{multiv_mix_Pois}. Suppose additionally that
\begin{equation}
\label{E_cond}
\sup_{n\in\mathbb N}p_r^{(n)}\mathbb E X_r^{(n)}<\infty,\quad r=1,\ldots,s.
\end{equation}
	
Then
\[
\bigl(p_1^{(n)}X_1^{(n)},\ldots,p_s^{(n)}X_s^{(n)}\bigr)\xrightarrow d\bigl(Y_1,\ldots,Y_s\bigr)\quad\text{ as $n\to\infty$}.
\]
\end{theorem}

\begin{proof}
Denote by $\mathcal G_n$ the probability generating function of $\mathbf X^{(n)}$:
\begin{equation}
\label{G_n}
\mathcal G_n(u_1,\ldots,u_s)=\mathbb E\Bigl(u_1^{X_1^{(n)}}\cdot\ldots\cdot u_s^{X_s^{(n)}}\Bigr),
\end{equation}
which is well defined at least for $(u_1,\ldots,u_s)\in[0,1]^s$. Then, by \eqref{thin}, the probability generating function of $\bigl(p_1^{(n)}\odot X_1^{(n)},\ldots,p_s^{(n)}\odot X_s^{(n)}\bigr)$ takes the form
\begin{equation}
\begin{aligned}
\label{pgf_px}
&\hspace{15.2 pt}\mathbb E\Bigl(u_1^{p_1^{(n)}\odot X_1^{(n)}}\cdot
\ldots\cdot u_s^{p_s^{(n)}\odot X_s^{(n)}}\Bigr)
=\mathbb E
\biggl(u_1^{\sum_{i=1}^{X_1^{(n)}}\varepsilon_{1,i}^{(n)}}\cdot
\ldots\cdot u_s^{\sum_{i=1}^ {X_s^{(n)}}\varepsilon_{s,i}^{(n)}}\biggr)
\\&=\quad\hspace{-7 pt}\mathbb E\sum_{k_1,\ldots,k_s=0}^\infty
u_1^{\sum_{i=1}^{k_1}\varepsilon_{1,i}^{(n)}}\cdot
\ldots\cdot u_s^{\sum_{i=1}^ {k_s}\varepsilon_{s,i}^{(n)}}\cdot
\mathds{1}\bigl\{X_1^{(n)}=k_1,\ldots,X_s^{(n)}=k_s\bigr\}
\\&=\hspace{-7 pt}\sum_{k_1,\ldots,k_s=0}^\infty\Bigl(\mathbb E u_1^{\varepsilon_{1,1}^{(n)}}\Bigr)^{k_1}\cdot\ldots
\cdot\Bigl(\mathbb E u_s^{\varepsilon_{s,1}^{(n)}}\Bigr)^{k_s}\cdot
\mathbb P\bigl\{X_1^{(n)}=k_1,\ldots,X_s^{(n)}=k_s\bigr\}
\\&=\quad\hspace{-7 pt}\mathbb E\Bigl(\bigl(p_1^{(n)}u_1+q_1^{(n)}\bigr)^{X_1^{(n)}}\cdot
\ldots\cdot\bigl(p_s^{(n)}u_s+q_s^{(n)}\bigr)^{X_s^{(n)}}\Bigr)
\\&=\quad\hspace{-7 pt}
\mathcal G_n\bigl(p_1^{(n)}u_1+q_1^{(n)},\ldots,p_s^{(n)}u_s+q_s^{(n)}\bigr).
\end{aligned}
\end{equation}
Here $\varepsilon_{r,i}^{(n)}$, $r=1,\ldots,s$, $i\in\mathbb N$, are independent $\mathsf{Bin}(1,p_r^{(n)})$, and $q_r^{(n)}=1-p_r^{(n)}$.
	
Next, denote by $\mathcal L$ the Laplace transform of $\mathbf Y=(Y_1,\ldots,Y_s)$:
\begin{equation}
\label{L_Y}
\mathcal L(t_1,\ldots,t_s)=\mathbb E\mathrm e^{-t_1Y_1-\ldots-t_sY_s},\quad t_1,\ldots,t_s\ge0.
\end{equation}
Then the probability generating function of $\mathbf Z=(Z_1,\ldots,Z_s)$ becomes
\begin{equation}
\label{pgf_z}
\begin{aligned}
\mathbb E\bigl(u_1^{Z_1}\cdot\ldots\cdot u_s^{Z_s}\bigr)&=
\mathbb E\,\mathbb E_{\mathbf Y}\bigl(u_1^{Z_1}\cdot\ldots\cdot u_s^{Z_s}\bigr)
=\mathbb E\Bigl(\bigl(\mathbb E_{Y_1}u_1^{Z_1}\bigr)\cdot\ldots\cdot
\bigl(\mathbb E_{Y_s}u_s^{Z_s}\bigr)\Bigr)
\\&=\mathbb E\Bigl(\mathrm e^{-Y_1(1-u_1)}\cdot\ldots\cdot
\mathrm e^{-Y_s(1-u_s)}\Bigr)=\mathcal L(1-u_1,\ldots,1-u_s),
\end{aligned}
\end{equation}
where $\mathbb E_{\mathbf Y},\mathbb E_{Y_1},\ldots,\mathbb E_{Y_s}$ stand for conditional means with respect to $\mathbf Y,Y_1,\ldots,Y_s$, respectively.
So, by \eqref{px_to_z}, \eqref{pgf_px}, and \eqref{pgf_z},
\[\lim_{n\to\infty}\mathcal G_n\bigl(p_1^{(n)}u_1+q_1^{(n)},\ldots,
p_s^{(n)}u_s+q_s^{(n)}\bigr)=\mathcal L(1-u_1,\ldots,1-u_s),
\quad u_1,\ldots,u_s\in[0,1],\]
and thus
\begin{equation}
\label{know}
\lim_{n\to\infty}\mathcal G_n\bigl(1-p_1^{(n)}t_1,\ldots,
1-p_s^{(n)}t_s\bigr)=\mathcal L(t_1,\ldots,t_s),
\quad t_1,\ldots,t_s\in[0,1].
\end{equation}
	
It follows from the multivariate mean value theorem that
\begin{equation}
\begin{aligned}
\label{G-G}
\Bigl|\mathcal G_n\bigl(1-p_1^{(n)}t_1,\ldots,1-&p_s^{(n)}t_s\bigr)-
\mathcal G_n\bigl(\mathrm e^{-p_1^{(n)}t_1},\ldots,\mathrm e^{-p_s^{(n)}t_s}\bigr)\Bigr|
\\&\le\sum_{r=1}^s\sup_{u_1,\ldots,u_s\in[0,1]}
\frac{\partial \mathcal G_n}{\partial u_r}(u_1,\ldots,u_s)\Bigl(\mathrm e^{-p_r^{(n)}t_r}-1+p_r^{(n)}t_r\Bigr).
\end{aligned}
\end{equation}
Since $\mathrm e^{-p_r^{(n)}t_r}-1+p_r^{(n)}t_r={\scriptstyle\mathcal O}\bigl(p_r^{(n)}\bigr)$ as $n\to\infty$, and
\[p_r^{(n)}\sup_{u_1,\ldots,u_s\in[0,1]}\frac{\partial \mathcal G_n}{\partial u_r}(u_1,\ldots,u_s)=
p_r^{(n)}\frac{\partial \mathcal G_n}{\partial u_r}(1,\ldots,1)=p_r^{(n)}\mathbb EX_r^{(n)}=\mathcal O(1)\]
by \eqref{E_cond}, then the right-hand side of \eqref{G-G} is ${\scriptstyle\mathcal O}(1)$ as $n\to\infty$. Thus, \eqref{know} implies
\[\lim_{n\to\infty}\mathcal G_n\bigl(\mathrm e^{-p_1^{(n)}t_1},\ldots,\mathrm e^{-p_s^{(n)}t_s}\bigr)=\mathcal L(t_1,\ldots,t_s),
\quad t_1,\ldots,t_s\in[0,1].\]
Due to \eqref{G_n} and \eqref{L_Y}, the left-hand and right-hand sides of the preceding relation are Laplace transforms of $\bigl(p_1^{(n)}X_1^{(n)},\ldots,p_s^{(n)}X_s^{(n)}\bigr)$ and $\bigl(Y_1,\ldots,Y_s\bigr)$, respectively. So, the claim follows from Theorem 2 in \cite{Yak}.
\end{proof}

\section{Proof of Theorem \ref{main_th}}
\label{proof_sec}

In this short section, we apply Theorems \ref{conv_th} and \ref{trick_th} to the proof of Theorem \ref{main_th}. Let
\begin{equation}
\label{TV}
\begin{gathered}
\tilde T_0^{(n)}=\inf\bigl\{x\in\mathbb R\colon
H^{(n)}\bigl(\{0\}\times(x,+\infty)\bigr)=0\bigr\},\\
V_r^{(n)}=H^{(n)}\bigl(\{r\}\times\bigl(\tilde T_0^{(n)},+\infty\bigr)\bigr)
\end{gathered}
\end{equation}
for $r\in\mathbb N$. In the notation of Section \ref{prel_sec}, $\tilde T_0^{(n)}=\max_{i\le n}\psi^{(n)}\bigl(Z_{i,0}^{(n)}\bigr)$, and thus means the (centered and normalized) time when the main collector completes his album in the poissonized scheme. Next, $V_r^{(n)}$ is the number of points of the thinned process $T_{\ln^{-r}n}\eta_r^{(n)}$ to the right of $\tilde T_0^{(n)}$. So, according to \eqref{UZ}, it is easily seen to be equal in distribution to $\ln^{-r}n\odot U_r^{(n)}$, where the operation $\odot$ is defined in \eqref{thin}.

By Theorem \ref{conv_th} and the Skorokhod coupling (see, e.g., \cite{Res07}, p.~41), we may assume that $H^{(n)}\xrightarrow{v}H$ a.s.\ as $n\to\infty$. Then, by Theorem 3.13 in \cite{Res87}, $V_r^{(n)}\to V_r$ a.s., where $V_r$ (and $\tilde T_0$) is defined similarly to \eqref{TV} but with $H$ instead of $H^{(n)}$. It is here that we use the fact that $H^{(n)}(\{r\}\times\cdot)$ are given on the semi-compactified space $\mathbb R\cup\{+\infty\}$, because we need $(x,+\infty)$ to be relatively compact.

Let $(N_r(t),t\ge0)$, $r\in\mathbb N_0$, be independent unit-rate Poisson counting processes. Denote by $E$ the first jump time of $N_0$, and note that
$E\sim\mathsf{Exp}(1)$. Due to Remarks \ref{ind_PPP} and \ref{nhg_PPP},
$\tilde T_0\overset d=-\ln E$, and $(V_1,\ldots,V_s)\overset d=(N_r(E/r!),
r=1,\ldots,s)$ for any $s\in\mathbb N$.
Summarizing all the above, we have
\[\bigl(\ln^{-r}n\odot U_r^{(n)},\,r=1,\ldots,s\bigr)\xrightarrow d (N_r(E/r!),\,r=1,\ldots,s),\qquad n\to\infty.\]

Then, applying Theorem \ref{trick_th} and taking into account Remark \ref{trick_rem}, we obtain
\[\bigl(\ln^{-r}n\cdot U_r^{(n)},\,r=1,\ldots,s\bigr)\xrightarrow d (E/r!,\,r=1,\ldots,s),\qquad n\to\infty,\]
condition \eqref{E_cond} being satisfied due to \eqref{EU} and \eqref{hn_as}. Thus, the convergence in \eqref{main_th_eq} holds in the sense of finite-dimensional distributions. To complete the proof, it only remains to note that in $\mathbb R^{\infty}$ the notions of finite-dimensional convergence and convergence in distribution are equivalent (see, e.g., \cite{Res07}, pp.~53--54).\qed

\section{Convergence of associated point processes II.
\\An \textit r-dependent centering}
For the proofs of Theorems \ref{second_th} and \ref{third_th}, we will also use a specially constructed sequence of point processes. This time, however, their convergence will be achieved not at the expense of thinning, as in Section \ref{sec_conv_I}, but due to an $r$-dependent centering.

On the metric space $(\mathbb X,d)$ given by \eqref{X} and \eqref{d}, consider the point processes $\hat H^{(n)}$, $n\in\mathbb N$, defined as follows: for $B_r\in\mathfrak B\bigl(\mathbb R\cup\{+\infty\}\bigr)$, $r\in\mathbb N_0$, let
\begin{equation}
\label{Hh}
\hat H^{(n)}\Bigl(\bigcup_{r=0}^\infty\bigl(\{r\}\times B_r\bigr)\Bigr)=\sum_{r=0}^\infty\hat\eta_r^{(n)}(B_r),
\end{equation}
where
\[\hat\eta_r^{(n)}=\sum_{i=1}^n\delta_{\hat\psi_r^{(n)}\bigl(Z_{i,r}^{(n)}\bigr)},\]
and the centering/normalizing functions $\hat\psi_r^{(n)}$, slightly generalizing $\psi_r^{(n)}$ in \eqref{psi_r}, are defined as follows:
\begin{equation}
\hat\psi_r^{(n)}(x)=\frac xn-\ln n-r\ln\ln n-\delta_r^{(n)},
\quad x\in\mathbb R.\label{hpsi_r}
\end{equation}
Here, for each $r\in\mathbb N_0$, $\delta_r^{(n)}$ is a fixed numerical sequence such that $\lim_{n\to\infty}\delta_r^{(n)}=0$. We will also use counterparts of $\hat H^{(n)}$ in the original, non-poissonized scheme:
\begin{equation*}
\hat\Xi^{(n)}\Bigl(\bigcup_{r=0}^\infty\bigl(\{r\}\times B_r\bigr)\Bigr)=\sum_{r=0}^\infty\hat\xi_r^{(n)}(B_r),
\end{equation*}
where
\[\hat\xi_r^{(n)}=\sum_{i=1}^n\delta_{\hat\psi_r^{(n)}\bigl(Y_{i,r}^{(n)}\bigr)},\]
and $Y_{i,r}^{(n)}$ are defined at the beginning of Section \ref{prel_sec}.

The following result is an analogue of Theorem \ref{conv_th} but differs in two significant aspects. Firstly, we now use the $r$-dependent centering \eqref{hpsi_r} instead of thinning, and secondly, the original scheme instead of the poissonized one. The latter will require a special depoissonization procedure based on the coupling formula \eqref{ZY} and similar to that used in \cite{Il19}. Unlike Theorem \ref{conv_th}, we need to consider this result in the original setting rather than in the poissonized one due to the lack of any counterpart to \eqref{UZ} for~$\hat U_r^{(n)}$.

\begin{theorem}
\label{conv_th_2}
Let $H$ be defined as in Theorem \ref{conv_th}.
Then $\hat\Xi^{(n)}\xrightarrow{vd}H$ as $n\to\infty$.
\end{theorem}

\begin{remark}
This theorem may be regarded as an infinite-dimensional extension of Theorem 3.1 in \cite{Il19}.
\end{remark}

\begin{proof}[Proof of Theorem \ref{conv_th_2}]
To begin with, we will return for a while to the poissonized scheme and prove that $\hat H^{(n)}\xrightarrow{vd}H$ as $n\to\infty$. Similarly to \eqref{cond_a} and \eqref{cond_b}, it is enough only to show that
\begin{gather}
\lim_{n\to\infty}\mathbb P\{\hat H^{(n)}(U)=0\}=\mathbb P\{H(U)=0\},
\label{cond_a2}\\
\lim_{n\to\infty}\mathbb E\hat H^{(n)}(U)=\mathbb EH(U),
\label{cond_b2}
\end{gather}
for each $U\in\mathcal U$, where the ring $\mathcal U$ is defined in \eqref{bound_set}. Moreover, as before we may assume that $U\subset\mathbb N_0\times\mathbb R$. Then, by \eqref{Hh}
\begin{equation}
\begin{aligned}
\label{phh}
\mathbb P\{\hat H^{(n)}(U)=0\}&=\mathbb P\Bigl\{\hat H^{(n)}\Bigl(\bigcup_{r=0}^s\bigl(\{r\}\times B_r\bigr)\Bigr)=0\Bigr\}\\&=\mathbb P\Bigl\{\sum_{r=0}^s \hat\eta_r^{(n)}(B_r)=0\Bigr\}=
\mathbb P\bigl\{\hat\eta_r^{(n)}(B_r)=0\text{ for }r=0,\ldots,s\bigr\}.
\end{aligned}
\end{equation}
Hence, using the i.i.d.~property of $Z_{i,r}^{(n)}$ and inclusion-exclusion, we have
\begin{equation*}
\mathbb P\{\hat H^{(n)}(U)=0\}=\Bigl(1-\sum_{0\le r_1\le s}
\hat P_{r_1}^{(n)}+\sum_{0\le r_1<r_2\le s}\hat P_{r_1,r_2}^{(n)}-
\ldots+(-1)^{s+1}\hat P_{0,1,\ldots,s}^{(n)}\Bigr)^n,
\end{equation*}
where the probabilities $\hat P_{r_1,\ldots,r_m}^{(n)}$ are defined similarly to $P_{r_1,\ldots,r_m}^{(n)}$ in \eqref{P}, but with $\hat\psi_r^{(n)}$ instead of $\psi^{(n)}$. So,
\begin{multline*}
\lim_{n\to\infty}\ln\mathbb P\{\hat H^{(n)}(U)=0\}\\=-\sum_{0\le r_1\le s}\lim_{n\to\infty}n\hat P_{r_1}^{(n)}+
\sum_{0\le r_1<r_2\le s}\lim_{n\to\infty}n\hat P_{r_1,r_2}^{(n)}-\ldots+(-1)^{s+1}\lim_{n\to\infty}
n\hat P_{0,1,\ldots,s}^{(n)}.
\end{multline*}
As before, in order to prove \eqref{cond_a2}, it suffices to show that the limits in the first sum equal
$\frac 1{r_1!}\int_{B_{r_1}}\mathrm e^{-x}\,\mathrm dx$,
and those in the second sum vanish.

Using \eqref{f_1} and \eqref{f_2}, we may easily calculate the densities $\hat f_{r_1}^{(n)}$ and $\hat f_{r_1,r_2}^{(n)}$ of $\hat\psi_{r_1}^{(n)}\bigl(Z_{i,r_1}^{(n)}\bigr)$ and $\bigl(\hat\psi_{r_1}^{(n)}\bigl(Z_{i,r_1}^{(n)}\bigr),
\hat\psi_{r_2}^{(n)}\bigl(Z_{i,r_2}^{(n)}\bigr)\bigr)$, respectively:
\begin{gather*}
\hat f_{r_1}^{(n)}(x)=\frac{\bigl(\frac x{\ln n}+\frac{r_1\ln\ln n}{\ln n}+\frac{\delta_{r_1}^{(n)}}{\ln n}+1\bigr)^{r_1}\mathrm e^{-x-\delta_{r_1}^{(n)}}}{n\,r_1!}\cdot\mathds 1\bigl\{-\ln n-r_1\ln\ln n-\delta_{r_1}^{(n)}\le x\bigr\},\quad x\in\mathbb R,\\
\begin{aligned}
\hat f_{r_1,r_2}^{(n)}&(x,y)=\frac{\bigl(\frac x{\ln n}+\frac{r_1\ln\ln n}{\ln n}+\frac{\delta_{r_1}^{(n)}}{\ln n}+1\bigr)^{r_1}\mathrm e^{-y-\delta_{r_2}^{(n)}}}{n\,r_1!\,(r_2-r_1-1)!}\\&\times\bigl(y-x+(r_2-r_1)\ln\ln n+\bigl(\delta_{r_2}^{(n)}-\delta_{r_1}^{(n)}\bigr)\bigr)^{r_2-r_1-1}
(\ln n)^{-(r_2-r_1)}\\&\times\mathds 1\bigl\{-\ln n-r_2\ln\ln n-\delta_{r_2}^{(n)}\le x-(r_2-r_1)\ln\ln n-\bigl(\delta_{r_2}^{(n)}-\delta_{r_1}^{(n)}\bigr)\le y\bigr\},
\quad x,y\in\mathbb R.
\end{aligned}
\end{gather*}
Hence,
\begin{equation}
\begin{aligned}
\label{2-lim}
n\hat P_{r_1}^{(n)}&=n\int_{B_{r_1}}\hat f_{r_1}^{(n)}(x)\,\mathrm dx=
\frac 1{r_1!}\int_{B_{r_1}}\Bigl(\frac x{\ln n}+\frac{r_1\ln\ln n}{\ln n}+\frac{\delta_{r_1}^{(n)}}{\ln n}+1\Bigr)^{r_1}\mathrm e^{-x-\delta_{r_1}^{(n)}}\\&\times\mathds 1\{-\ln n-r_1\ln\ln n-\delta_{r_1}^{(n)}\le x\}\,\mathrm dx\to\frac 1{r_1!}\int_{B_{r_1}}
\mathrm e^{-x}\,\mathrm dx\quad\text{as $n\to\infty$}
\end{aligned}
\end{equation}
by dominated convergence since $B_{r_1}$ is bounded from below.

The proof that $\lim_{n\to\infty}n\hat P_{r_1,r_2}^{(n)}=0$ is a bit more technical than a similar piece in the proof of Theorem \ref{conv_th}. Letting $\alpha=\inf B_{r_1}$, $\beta=\inf B_{r_2}$, and noting that $\alpha,\beta>-\infty$, we have
\begin{align*}
n\hat P_{r_1,r_2}^{(n)}&\le n\int_\alpha^{+\infty}\int_\beta^{+\infty}
\hat f_{r_1,r_2}^{(n)}(x,y)\,\mathrm dx\,\mathrm dy\\&=
\frac{(\ln n)^{-(r_2-r_1)}}{r_1!\,(r_2-r_1-1)!}\int_\alpha^{+\infty}\int_\beta^{+\infty}\Bigl(\frac x{\ln n}+\frac{r_1\ln\ln n}{\ln n}+\frac{\delta_{r_1}^{(n)}}{\ln n}+1\Bigr)^{r_1}\\&\times\bigl(y-x+(r_2-r_1)\ln\ln n+\bigl(\delta_{r_2}^{(n)}-\delta_{r_1}^{(n)}\bigr)\bigr)^{r_2-r_1-1}\mathrm e^{-y-\delta_{r_2}^{(n)}}\\&\times\mathds 1\bigl\{-\ln n-r_2\ln\ln n-\delta_{r_2}^{(n)}\le x-(r_2-r_1)\ln\ln n-\bigl(\delta_{r_2}^{(n)}-\delta_{r_1}^{(n)}\bigr)\le y\bigr\}\,\mathrm dx\,\mathrm dy.
\end{align*}
Setting $z=y-x+(r_2-r_1)\ln\ln n+\bigl(\delta_{r_2}^{(n)}-\delta_{r_1}^{(n)}\bigr)$, we then obtain
\begin{align*}
n\hat P_{r_1,r_2}^{(n)}&\le\frac 1{r_1!\,(r_2-r_1-1)!}\int_\alpha^{+\infty}\int_0^{+\infty}
\Bigl|\frac x{\ln n}+\frac{r_1\ln\ln n}{\ln n}+\frac{\delta_{r_1}^{(n)}}{\ln n}+1\Bigr|^{r_1}
\\&\times z^{r_2-r_1-1}\mathrm e^{-(x+z+\delta_{r_1}^{(n)})}\cdot\mathds 1\bigl\{x+z\ge\beta+(r_2-r_1)\ln\ln n+\bigl(\delta_{r_2}^{(n)}-\delta_{r_1}^{(n)}\bigr)\bigr\}\,\mathrm dx\,\mathrm dz.
\end{align*}
Again by dominated convergence, the right-hand side vanishes as $n\to\infty$. This completes the proof of \eqref{cond_a2}.

The proof of \eqref{cond_b2} is similar to that of \eqref{cond_b}. Analogously to \eqref{phh},
\begin{equation*}
\mathbb E\hat H^{(n)}(U)=\mathbb E\hat H^{(n)}\Bigl(\bigcup_{r=0}^s\bigl(\{r\}\times B_r\bigr)\Bigr)=\sum_{r=0}^s\mathbb E\hat\eta_r^{(n)}(B_r).
\end{equation*}
Since $\hat\eta_r^{(n)}(B_r)\sim\mathsf{Bin}\bigl(n,\hat P_r^{(n)}\bigr)$, \eqref{2-lim} yields
\[\mathbb E\hat H^{(n)}(U)=\sum_{r=0}^sn\hat P_r^{(n)}\xrightarrow[n\to\infty]{}\sum_{r=0}^s\frac 1{r!}\int_{B_r}
\mathrm e^{-x}\,\mathrm dx=\mathbb EH(U).\]
This proves \eqref{cond_b2}, which, along with \eqref{cond_a2}, delivers $\hat H^{(n)}\xrightarrow[n\to\infty]{vd}H$.

Our next goal is to carry out a depoissonization procedure which enables to turn $\hat H^{(n)}\xrightarrow{vd}H$ into $\hat\Xi^{(n)}\xrightarrow{vd}H$.
Let $\mathcal C$ stand for the ring of finite unions of disjoint closed segments on different levels of $\mathbb X$, which are bounded with respect to $d$ in \eqref{d}:
\begin{equation*}
\mathcal C=\Bigl\{\bigcup_{r=0}^s\bigcup_{k=1}^{l_r}\bigl(\{r\}
\times[a_{r,k},b_{r,k}]\bigr)\colon s,l_0,\ldots,l_s\in\mathbb N_0\Bigr\}.
\end{equation*}
Here $a_{r,k},b_{r,k}\in\mathbb R\cup\{+\infty\}$, and we will always use the convention that $+\infty\pm\varepsilon=+\infty$. In the terminology of \cite{Kal}, $\mathcal C$ is a dissecting ring in $(\mathbb X,d)$.

Fix any $C\in\mathcal C$. By Lemma 3.4 in \cite{Il19}, which was proved on the basis of the coupling formula \eqref{ZY}, for any $n\in\mathbb N$, $\varepsilon>0$, and all $r$, $k$ we have
\begin{multline*}
\mathbb P\bigl\{\hat\xi_r^{(n)}\bigl([a_{r,k},b_{r,k}]\bigr)\ne
\hat\eta_r^{(n)}\bigl([a_{r,k},b_{r,k}]\bigr)\bigr\}\\\le c_r\varepsilon^{-4}n^{-1}+\mathbb P\bigl\{\hat\eta_r^{(n)}\bigl([a_{r,k}-\varepsilon,a_{r,k}+\varepsilon]
\bigr)\ge1\bigr\}+\mathbb P\bigl\{\hat\eta_r^{(n)}\bigl([b_{r,k}-\varepsilon,b_{r,k}+\varepsilon]
\bigr)\ge1\bigr\}
\end{multline*}
with some $c_r>0$. (In fact, this lemma was proved for the centering/normalizing function $\psi_r^{(n)}$, not $\hat\psi_r^{(n)}$, and for $a_{r,k},b_{r,k}\in\mathbb R$, not $\mathbb R\cup\{+\infty\}$, but its proof remains valid also in our case.) Taking $\varepsilon=n^{-\frac 15}$, we then obtain:
\begin{equation}
\begin{aligned}
\label{Pne}
&\mathbb P\bigl\{\hat\xi_r^{(n)}\bigl([a_{r,k},b_{r,k}]\bigr)\ne
\hat\eta_r^{(n)}\bigl([a_{r,k},b_{r,k}]\bigr)\bigr\}\le c_rn^{-\frac 15}\\
+\,&\mathbb P\bigl\{\hat\eta_r^{(n)}\bigl([a_{r,k}-n^{-\frac 15},a_{r,k}+n^{-\frac 15}]
\bigr)\ge1\bigr\}+\mathbb P\bigl\{\hat\eta_r^{(n)}\bigl([b_{r,k}-n^{-\frac 15},b_{r,k}+n^{-\frac 15}]
\bigr)\ge1\bigr\}.
\end{aligned}
\end{equation}

Let $\eta_r$, $r\in\mathbb N_0$, stand for the Poisson point process on $\mathbb R\cup\{+\infty\}$ with intensity measure $\lambda_r$ given by
\begin{equation*}
\lambda_r(B)=\frac 1{r!}\int_B\mathrm e^{-x}\,\mathrm dx,\quad B\in\mathfrak B\bigl(\mathbb R\cup\{+\infty\}\bigr).
\end{equation*}
Actually, $\eta_r$ are the single-level point processes which make up the limiting multilevel process $H$.
Since $\hat H^{(n)}\xrightarrow[n\to\infty]{vd}H$ implies $\hat \eta_r^{(n)}\xrightarrow[n\to\infty]{vd}\eta_r$, we get
\begin{equation*}
\mathbb P\bigl\{\hat\eta_r^{(n)}\bigl([a_{r,k}-n^{-\frac 15},a_{r,k}+n^{-\frac 15}]
\bigr)\ge1\bigr\}\xrightarrow[n\to\infty]{}\mathbb P\bigl\{\eta_r\bigl(\{a_{r,k}\}\bigr)\ge1\bigr\}=0,
\end{equation*}
and the same holds for $b_{r,k}$.
Here the simultaneous passage to the limit in the measure and its argument is justified by Lemma 3.5 in \cite{Il19}. Hence, by \eqref{Pne}
\[\mathbb P\bigl\{\hat\xi_r^{(n)}\bigl([a_{r,k},b_{r,k}]\bigr)\ne
\hat\eta_r^{(n)}\bigl([a_{r,k},b_{r,k}]\bigr)\bigr\}\to0,\qquad n\to\infty.\]
So, by the coupling inequality we have for any $m\in\mathbb N_0$
\begin{multline*}
\bigl|\mathbb P\bigl\{\hat\Xi^{(n)}(C)=m\bigr\}-\mathbb P\bigl\{\hat H^{(n)}(C)=m\bigr\}\bigr|\le\mathbb P\bigl\{\hat\Xi^{(n)}(C)\ne\hat H^{(n)}(C)\bigr\}\\
\le\sum_{r=0}^s\sum_{k=1}^{l_r}\mathbb P\bigl\{\hat\xi_r^{(n)}\bigl([a_{r,k},b_{r,k}]\bigr)\ne
\hat\eta_r^{(n)}\bigl([a_{r,k},b_{r,k}]\bigr)\bigr\}\to0,\qquad n\to\infty.
\end{multline*}
Thus, it follows from $\hat H^{(n)}\xrightarrow[n\to\infty]{vd}H$ that
\[\lim_{n\to\infty}\mathbb P\bigl\{\hat\Xi^{(n)}(C)=m\bigr\}=
\lim_{n\to\infty}\mathbb P\bigl\{\hat H^{(n)}(C)=m\bigr\}=
\mathbb P\bigl\{H(C)=m\bigr\}\]
for any $C\in\mathcal C$ and $m\in\mathbb N_0$. Since $\mathcal C$ is a dissecting ring and the limiting point process $H$ is simple, Theorem 4.15 in \cite{Kal} yields $\hat\Xi^{(n)}\xrightarrow[n\to\infty]{vd}H$. This completes the proof.
\end{proof}

\section{Proof of Theorems \ref{second_th} and \ref{third_th}}

In this section, we apply Theorem \ref{conv_th_2} to the proof of Theorems \ref{second_th} and \ref{third_th}. As in Section \ref{proof_sec}, it suffices only to prove the convergence of finite-dimensional distributions. 

Similarly to \eqref{TV}, let
\begin{gather}
\label{hTr}
\hat T_r^{(n)}=\inf\bigl\{x\in\mathbb R\colon
\hat\Xi^{(n)}\bigl(\{r\}\times(x,+\infty)\bigr)=0\bigr\},\\
\nonumber
\hat V_r^{(n)}=\hat\Xi^{(n)}\bigl(\{r\}\times\bigl(\hat T_0^{(n)},+\infty\bigr)\bigr)
\end{gather}
for $r\in\mathbb N_0$. Particularly, $\hat V_0^{(n)}=0$ a.s. By definition of $\hat\Xi^{(n)}$ and $\hat\xi_r^{(n)}$, and according to \eqref{T}, it holds that
\begin{equation}
\label{hT}
\hat T_r^{(n)}=\max_{i\le n}\hat\psi_r^{(n)}\bigl(Y_{i,r}^{(n)}\bigr)=\hat\psi_r^{(n)}
\bigl(T_r^{(n)}\bigr).
\end{equation}
Thus, $\hat T_r^{(n)}$ means the (centered and normalized) time when the $r^{th}$ collector completes his album in the original, non-poissonized scheme. Next, $\hat V_r^{(n)}$ is the number of points of the single-level process $\hat\xi_r^{(n)}$ to the right of $\hat T_0^{(n)}$. By definition of $\hat\xi_r^{(n)}$ and according to \eqref{hpsi_r}, \eqref{hT}, we have
\begin{equation*}
\hat V_r^{(n)}=\card\bigl\{i\colon Y_{i,r}^{(n)}>T_0^{(n)}+rn\ln\ln n+n\bigl(\delta_r^{(n)}-\delta_0^{(n)}\bigr)\bigr\}.
\end{equation*}
Comparing this with \eqref{hU}, we can see that, with any choice of ${\scriptstyle\mathcal{O}}(n)$ in \eqref{d_r}, one may choose $\delta_r^{(n)}$, $r\in\mathbb N_0$, in such a way that
$\hat V_r^{(n)}=\hat U_r^{(n)}$.

Just as it was done in Section \ref{proof_sec}, Theorem \ref{conv_th_2} and the Skorokhod coupling imply that, for any $s\in\mathbb N$,
\begin{equation*}
\bigl(\hat U_r^{(n)},\,r=1,\ldots,s\bigr)\xrightarrow d (N_r(E/r!),\,r=1,\ldots,s),\qquad n\to\infty,
\end{equation*}
where $E\sim\mathsf{Exp}(1)$, and $N_r$ are unit-rate Poisson counting processes, independent of each other and of $E$. This completes the proof of \eqref{second_th_eq}.
The generating function of the random vector on the right-hand side can be easily calculated by means of conditioning with respect to $E$:
\begin{multline*}
\mathbb E\Bigl(\prod_{r=1}^su_r^{N_r(E/r!)}\Bigr)=\mathbb E\Bigl(\prod_{r=1}^s\mathbb E_E\bigl(u_r^{N_r(E/r!)}\bigr)\Bigr)=
\mathbb E\Bigl(\prod_{r=1}^s\mathrm e^{-E(1-u_r)/r!}\Bigr)\\=
\int_0^{+\infty}\mathrm e^{-x-x\sum_{r=1}^s\frac{1-u_r}{r!}}\,\mathrm dx=
\Bigl(1+\sum_{r=1}^s\frac{1-u_r}{r!}\Bigr)^{-1},\qquad (u_1,\ldots,u_s)\in[0,1]^s.
\end{multline*}
Letting here $s\to\infty$ yields \eqref{Gen_G}.

Finally, again Theorem \ref{conv_th_2} and the Skorokhod coupling, combined with \eqref{hT}, \eqref{hTr}, and Remarks \ref{ind_PPP}, \ref{nhg_PPP} imply that, for any $s\in\mathbb N_0$,
\begin{multline*}
\bigl(\hat\psi_r^{(n)}\bigl(T_r^{(n)}\bigr),r=0,\ldots,s\bigr)=\bigl(\hat T_r^{(n)},r=0,\ldots,s\bigr)\\\xrightarrow d
\bigl(-\ln r!-\ln E_r,r=0,\ldots,s\bigr),\qquad n\to\infty,
\end{multline*}
where $E_r$ are independent $\mathsf{Exp}(1)$-distributed random variables.
Here $E_r$ actually means the first jump time of $N_r$. Thus, $-\ln r!-\ln E_r$ on the right-hand side, by Remark \ref{nhg_PPP}, is the rightmost point of $H\bigl(\{r\}\times\cdot\bigr)$, the $r^{th}$ level of the limiting process $H$. To complete the proof of \eqref{third_th_eq}, it suffices to note that $-\ln r!-\ln E_r\overset d=B_r$ from \eqref{Br}, and to put in \eqref{hpsi_r} $\delta_r^{(n)}=0$.\qed

\section{Concluding remarks and open problems}

In Theorems \ref{main_th} and \ref{second_th}, the reference point is $T_0^{(n)}$, the time when the main collector completed his album. That is, we studied the number of empty spots in the album of the $r^{th}$ brother at time $T_0^{(n)}$ in Theorem \ref{main_th} and at time $T_0^{(n)}+d_r^{(n)}$ in Theorem \ref{second_th}. Similarly, we could fix some $r_0\in\mathbb N_0$ and consider the number of empty spots in the collection of the $r^{th}$ brother at times $T_{r_0}^{(n)}$ and $T_{r_0}^{(n)}+d_{r,r_0}^{(n)}$ with some $d_{r,r_0}^{(n)}\to+\infty$ as $n\to\infty$. The former makes sense only for $r>r_0$, while, in the latter case, we may consider $r<r_0$ with $d_{r,r_0}^{(n)}<0$, replacing $d_{r,r_0}^{(n)}\to+\infty$ by $d_{r,r_0}^{(n)}\to-\infty$. Repeating the above proofs with minimal changes, we can derive analogous results with similar limiting distributions (exponential and geometric, respectively).

Another possible reference point is $T_0^{(n)}(m)$, $m\in\mathbb N_0$, the time when the main collector first assembled some $n-m$ (unspecified) of $n$ coupons, or $T_{r_0}^{(n)}(m)$, the similar value for the $r_0^{\,\,th}$ collector. In this notation, $T_{r_0}^{(n)}(0)=T_{r_0}^{(n)}$.
In this case, we come to more general limiting distributions --- gamma and negative binomial, respectively. This can be proved along the above lines.

Let us now dwell on some possible directions for further research. Computer simula\-tion shows that the rate of convergence in Theorems \ref{main_th}, \ref{second_th}, and \ref{third_th} is rather slow. It would be interesting to obtain some upper bounds for these rates in terms of Wasserstein (for Theorems \ref{main_th} and \ref{third_th}) and total variation (for Theorem \ref{second_th}) distances. The most natural tool for this is, of course, the Stein method. Some of its applications to the coupon collector's problem are given in Chapter 6 of \cite{BHJ} and in \cite{Ross11}.

Also of interest is the following problem. Let $W_0^{(n)}$ denote the number of the youngest brother who managed to start his collection while the main collector was assembling his own. Thus,
\begin{equation*}
W_0^{(n)}=\max\bigl\{r\in\mathbb N_0\colon\min_{i\le n}Z_{i,r}^{(n)}
<\max_{i\le n}Z_{i,0}^{(n)}\bigr\}.
\end{equation*}
What is the limiting distribution for $W_0^{(n)}$? And what normalization is needed for this? Clearly, we may ask the same questions about $W_{r_0}^{(n)}$,
which is defined similarly but with $\max_{i\le n}Z_{i,r_0}^{(n)}$ instead of $\max_{i\le n}Z_{i,0}^{(n)}$. Most likely, these results could be obtained by means of extreme value theory.

Finally, a wide scope for research is provided by the extended coupon collector's problem with unequal probabilities. Some steps in this direction have been recently taken in \cite{DP18}.


\begin{thebibliography}{10}
	
	\bibitem{Omni}
	S.~{Abraham}, G.~{Brockman}, S.~{Sapp}, and A.~P. {Godbole}.
	\newblock {Omnibus sequences, coupon collection, and missing word counts.}
	\newblock {\em {Methodol. Comput. Appl. Probab.}}, 15(2):363--378, 2013.
	
	\bibitem{AOR}
	I.~{Adler}, S.~{Oren}, and S.~M. {Ross}.
	\newblock {The coupon-collector's problem revisited.}
	\newblock {\em {J. Appl. Probab.}}, 40(2):513--518, 2003.
	
	\bibitem{BHJ}
	A.~D. {Barbour}, L.~{Holst}, and S.~{Janson}.
	\newblock {\em {Poisson approximation.}}
	\newblock Oxford: Clarendon Press, 1992.
	
	\bibitem{BN52}
	G.~E. Bates and J.~Neyman.
	\newblock Contribution to the theory of accident proneness. {II}: True or false
	contagion.
	\newblock {\em Univ. California Publ. Stat.}, 1(10):255--276, 1952.
	
	\bibitem{BH}
	A.~{Boneh} and M.~{Hofri}.
	\newblock {The coupon-collector problem revisited -- a survey of engineering
		problems and computational methods.}
	\newblock {\em {Commun. Stat., Stochastic Models}}, 13(1):39--66, 1997.
	
	\bibitem{DR96}
	P.~J. {Davy} and J.~C.~W. {Rayner}.
	\newblock {Multivariate geometric distributions.}
	\newblock {\em {Commun. Stat., Theory Methods}}, 25(12):2971--2987, 1996.
	
	\bibitem{DP18}
	A.~V. {Doumas} and V.~G. {Papanicolaou}.
	\newblock {The siblings of the coupon collector.}
	\newblock {\em {Theory Probab. Appl.}}, 62(3):444--470, 2018.
	
	\bibitem{ER}
	P.~{Erd\H{o}s} and A.~{R\'enyi}.
	\newblock {On a classical problem of probability theory.}
	\newblock {\em {Publ. Math. Inst. Hung. Acad. Sci., Ser. A}}, 6:215--220, 1961.
	
	\bibitem{FLT}
	A.~Ferrari, G.~Letac, and J.-Y. Tourneret.
	\newblock Multivariate mixed {P}oisson distributions.
	\newblock In {\em 2004 12th European Signal Processing Conference}, pages
	1067--1070. IEEE, 2004.
	
	\bibitem{Foa01}
	D.~{Foata}, G.-N. {Han}, and B.~{Lass}.
	\newblock {Les nombres hyperharmoniques et la fratrie du collectionneur de
		vignettes.}
	\newblock {\em {S\'emin. Lothar. Comb.}}, 47:b47a, 20, 2001.
	
	\bibitem{Foa03}
	D.~{Foata} and D.~{Zeilberger}.
	\newblock {The collector's brotherhood problem using the Newman-Shepp symbolic
		method.}
	\newblock {\em {Algebra Univers.}}, 49(4):387--395, 2003.
	
	\bibitem{Gran}
	J.~{Grandell}.
	\newblock {\em {Mixed Poisson processes.}}
	\newblock London: Chapman \& Hall, 1997.
	
	\bibitem{HJK}
	P.~{Harremo\"es}, O.~{Johnson}, and I.~{Kontoyiannis}.
	\newblock {Thinning, entropy, and the law of thin numbers.}
	\newblock {\em {IEEE Trans. Inf. Theory}}, 56(9):4228--4244, 2010.
	
	\bibitem{Hol86}
	L.~{Holst}.
	\newblock {On birthday, collectors', occupancy and other classical urn
		problems.}
	\newblock {\em {Int. Stat. Rev.}}, 54:15--27, 1986.
	
	\bibitem{Il19}
	A.~Ilienko.
	\newblock Convergence of point processes associated with coupon collector’s
	and dixie cup problems.
	\newblock {\em Electron. Commun. Probab.}, 24:9 pp., 2019.
	
	\bibitem{JK77}
	N.~L. {Johnson} and S.~{Kotz}.
	\newblock {\em {Urn models and their application. An approach to modern
			discrete probability theory.}}
	\newblock John Wiley \& Sons, Hoboken, NJ, 1977.
	
	\bibitem{Kal}
	O.~{Kallenberg}.
	\newblock {\em {Random measures, theory and applications.}}
	\newblock Cham: Springer, 2017.
	
	\bibitem{KP}
	M.~{Kuba} and A.~{Panholzer}.
	\newblock {On moment sequences and mixed Poisson distributions.}
	\newblock {\em {Probab. Surv.}}, 13:89--155, 2016.
	
	\bibitem{LP}
	G.~{Last} and M.~{Penrose}.
	\newblock {\em {Lectures on the Poisson process.}}
	\newblock Cambridge: Cambridge University Press, 2017.
	
	\bibitem{Luko}
	S.~N. Luko.
	\newblock The “coupon collector's problem” and quality control.
	\newblock {\em Quality Engineering}, 21(2):168--181, 2009.
	
	\bibitem{Pint}
	N.~{Pintacuda}.
	\newblock {Coupons collectors via the martingales.}
	\newblock {\em {Boll. Unione Mat. Ital., V. Ser., A}}, 17:174--177, 1980.
	
	\bibitem{Ren}
	A.~{R\'enyi}.
	\newblock {A characterization of Poisson processes.}
	\newblock {\em {Publ. Math. Inst. Hung. Acad. Sci.}}, 1:519--527, 1957.
	
	\bibitem{Res87}
	S.~I. {Resnick}.
	\newblock {\em {Extreme values, regular variation, and point processes.}}
	\newblock New York etc.: Springer-Verlag, 1987.
	
	\bibitem{Res07}
	S.~I. {Resnick}.
	\newblock {\em {Heavy-tail phenomena. Probabilistic and statistical modeling.}}
	\newblock New York: Springer, 2007.
	
	\bibitem{Ross11}
	N.~{Ross}.
	\newblock {Fundamentals of Stein's method.}
	\newblock {\em {Probab. Surv.}}, 8:210--293, 2011.
	
	\bibitem{SvH}
	F.~W. {Steutel} and K.~{van Harn}.
	\newblock {Discrete analogues of self-decomposability and stability.}
	\newblock {\em {Ann. Probab.}}, 7:893--899, 1979.
	
	\bibitem{Yak}
	A.~L. {Yakymiv}.
	\newblock {A generalization of the Curtiss theorem for moment generating
		functions.}
	\newblock {\em {Math. Notes}}, 90(6):920--924, 2011.
	
\end{thebibliography}
\end{document}